\documentclass[12pt,a4paper]{article}
\usepackage[utf8x]{inputenc}
\usepackage{amsmath}
\usepackage{amsfonts}
\usepackage{amssymb}
\usepackage{amsthm}
\usepackage{mathrsfs}
\usepackage{fancyhdr}
\usepackage{graphicx}
\usepackage[usenames,x11names]{xcolor}
\usepackage{arabtex}
\usepackage{framed}
\usepackage[top=2cm,bottom=2cm,margin=2cm]{geometry}
\usepackage{cancel}
\usepackage{array}   

\usepackage[dvipdfm,pdfstartview=FitH,colorlinks=true,linkcolor=Red4,urlcolor=Red4,%
filecolor=green,citecolor=red]{hyperref}

\title{On the invertibility in periodic ARFIMA models}
\author{\sc Amine AMIMOUR$^{1}$ \\
and\\	
\sc Karima BELAIDE$^{2}$ \\
Department of Mathematics\\ Applied Mathematics Laboratory\\
University of Bejaia, Bejaia. Algeria \\[1mm]
\href{mailto:amineamimour@gmail.com}{amineamimour@gmail.com}$^{1}$ \\[1mm]
\href{mailto:k\_tim2002@yahoo.fr}{k\_tim2002@yahoo.fr}$^{2}$ \\[1mm]
}
\date{}

\def\EMdash{\leavevmode\hbox to 10.6mm{\vrule height .63ex depth -.59ex
		width 10mm\hfill}}

\theoremstyle{plain}
\numberwithin{equation}{section}
\newtheorem{thm}{Theorem}[section]
\newtheorem{theorem}[thm]{Theorem}

\newtheorem{proposition}[thm]{Proposition}

\begin{document}
	\maketitle
	\begin{abstract}
	The present paper, characterizes the invertibility and causality conditions of a periodic ARFIMA (PARFIMA) models. We first, discuss the conditions in the multivariate case, by considering the corresponding p-variate stationary ARFIMA models. Second, we construct the conditions using the univariate case and we deduce a new infinite autoregressive representation for the PARFIMA model, the results are investigated through a simulation study.
	\end{abstract}
	\section{Introduction and notations}
	A long memory processes have long been studied in the literature. (See, e.g. Granger and Joyeux \cite{GJ1980} and Hosking \cite{H1981} for early work, Baillie \cite{B1996} for more background).
	Hosking \cite{H1981} proposed a purely fractionally differenced autoregressive moving average process $\left( X_{t}\right)_{t\in 
		\mathbb{Z}
	}$ denoted by ARFIMA$(0,d,0)$ or ARFIMA for short, given by the equation  
	\begin{equation}
	\label{eq0.1}
	(1-B)^{d}X_{t} = \varepsilon _{t},
	\end{equation}
	where $d$ is a real number which denote the lag of the ARFIMA process, {$\varepsilon_{t}$} is a sequence of independent and identically distributed (i,i,d) random variables with mean zero and finite variance $\sigma^{2}$, and $B$ is the back shift operator, such that $X_{t-j}=B^{j}X_{t}$.
	The invertibility concept of time series models, loosely says that the ARFIMA model is invertible when we are able to express the noise process {$\varepsilon_{t}$} as a convergent series of the observations {$X_{t}$}.
	The characteristic equation for a ARFIMA process is 
	\begin{equation}
	\label{eq0.2}
	\varepsilon _{t}=\underset{j=0}{\overset{\infty }{\sum }}\pi _{j}X_{t-j},
	\end{equation}
	where
	\begin{center}
		$\pi _{j}=\frac{\Gamma (-d+j)}{\Gamma (-d)\Gamma (j+1)}$.
	\end{center}
	$\Gamma (.)$ is a gamma function and it is given by 
	\begin{equation*}
	\Gamma (z)=\left\{ 
	\begin{array}{c}
	\int_{0}^{\infty }s^{z-1}e^{-s}ds,\text{ si }z>0 \\ 
	\infty , \ \ \ \ \ \ \ \ \ \ \ \ \ \ \ \text{ si }z=0%
	\end{array}%
	,\right. 
	\end{equation*}
	if $z<0,$ $\Gamma (z)$ is defined by the recurrence formula $z\Gamma
	(z)=\Gamma (z+1).$
	In order for a ARFIMA process to be invertible, the infinite sum of the coefficients $\pi _{j}$ must be absolutely summable. However, when the lag  $d >-\frac{1}{2}$,  the infinite sum of the coefficients $\pi _{j}$ is absolutely summable. Odaki\cite{O1993} noticed that the ARFIMA process is invertible even when $-1<d$. These conditions concerning the univariate case. Chung \cite{C2002} used a Vector ARFIMA model to define the invertibility condition of multivariate p-dimensional stationary process in the sense of Hosking \cite{H1981}.  Recently, the invertibility condition for the bivariate and multivariate case is achieved by Kechagias and Pipiras \cite{KP2015}, they extended the definition of the bivariate LRD of Robinson \cite{R2008}, to the multivariate case, concluding that the trigonometric power-law coefficients, can be used to construct new univarite and multivariate causal LRD series. 
	
	Since a several time series encountered in practice exhibit periodic autocorrelation structure, a fact that cannot be explained, by the classical seasonal models or by the time series models with parameters invariants in time. In fact, Franses and Ooms \cite{FO1997} considered the PARFIMA models defined by
		\begin{equation}
		\label{eq0.3}
		X_{t} = (1-B)^{-d_{s}}\varepsilon_{t},
		\end{equation}	
where, $ d_ {s} $ is the fractional parameter which can vary with season $s=1,2,3,4$, to study quarterly inflation in the UK, allowing the degree of fractional integration to vary with season, the series analyzed in this study concern the quarterly inflation rate verified over a 33-year period, after the presentation of the correlogram, they concluded that the autocorrelations do not decrease rapidly, which indicates long-memory behavior in the series. Moreover, if the transformation of the series
be performed by classical differentiation in the sense of Box Jenkins, the autocorrelations indicate that the differentiated series can be overdifferentiated, because the sum of the autocorrelations will be close to $-0.5$, then the observations suggest the utility possible of ARFIMA models. Therefore, for a more appropriate modeling, which takes into account
seasonal variation they have proposed to merge the presence of long memory and the periodic dynamics, into a new model. The key result of these authors is that this model compared with the periodic short-memory models and long-memory models with constant coefficients, not only provides a plausible informative description of the periodic series, but also, can be useful for out-of-sample forecasts. Amimour and Belaide \cite{AB2020-2} studied recently, the probabilistics property of PARFIMA models, in the sense that the long memory parameter varies periodically in time such that $d_{t+p}=d_{t}$ with period $p \in\mathbb{N^{*}};p>1$, denoted by PtvARFIMA and has the following stochastic equation 

\begin{center}
	\begin{equation}
	\label{eq0.4}
	(1-B)^{d_{i}}X_{i+pm}=\varepsilon
	_{i+pm},%
	\makeatletter
	\renewcommand\theequation{\thesection.\arabic{equation}}
	\@addtoreset{equation}{section}
	\makeatother%
	\end{equation}
\end{center}
where for all $t\in 
\mathbb{Z}
$, there exists $i=\left\{ 1,........,p\right\}$, $m\in 
\mathbb{Z}
,$ such that $t=i+pm$ and the variance is periodic in $t$ such that $\sigma^{2}_{t+pm}$ =$\sigma^{2}_{t}$, the authors constructed also a local asymptotic normality property for this model see \cite{AB2020}. Indeed, the PARFIMA models have received little attention from time series analysts due to these complexities, one of the main difficulties is to achieving a invertibility condition. Noting that the works cited above based only on the sufficient condition of invertibility and causality, that is\\
when $d_{i}>0$, the process (\ref{eq0.4}) is invertible and has an infinite
autoregressive representation is as follows

\begin{center}
	\begin{equation}
	\label{eq0.5}
	\varepsilon _{i+pm}=(1-B)^{d_{i}}X_{i+pm}=\underset{}{\overset{}{\overset{%
				\infty }{\underset{j=0}{\sum }}\text{ }\pi _{j}^{i}X_{i+pm-j}}},%
	\end{equation}
\end{center}
where 
\begin{center}
	$\pi _{j}^{i}=\frac{\Gamma (j-d_{i})}{\Gamma (j+1)\Gamma (-d_{i})},$
\end{center}
when $d_{i}<\frac{1}{2}$, the process (\ref{eq0.4}) is causal and has an infinite moving-average representation is as follows

\begin{center}
	\begin{equation}
	\label{eq0.6}
	X_{i+pm}=(1-B)^{-d_{i}}\varepsilon _{i+pm}=\overset{\infty }{\underset{j=0}{%
			\sum }}\psi _{j}^{i}\varepsilon _{i+pm-j},%
	\end{equation}
\end{center}
where
\begin{center}
	$\psi _{j}^{i}=\frac{\Gamma (j+d_{i})}{\Gamma (j+1)\Gamma (d_{i})}.$
\end{center}
\begin{equation}
\label{eq1.4}
\psi _{j}^{i}\sim v_{i}j^{d_{i}-1}, v_{i}>0, as\ j\rightarrow \infty .
\end{equation}
See propositions (2.1) and (2.2) in Amimour and belaide \cite{AB2020-2}. For instance, precise information on the invertibility conditions of a PtvARFIMA model is important to circumscribe a model's parameter space. Moreover, verifying invertibility in the necessary and sufficient conditions sense, is required for statistical inference, in this article we address the above topic. To find an invertibility condition of PtvARFIMA.

The work is organized as follows. We first discuss in section 2 some characteristics of our model using his associated stationary multivariate process. In section 3, we focus on the constructing of new infinite autoregressive representation of our model. We finish this work by the simulation for illustrating the results in section 4.

\section{PtvARFIMA and p-variate stationary ARFIMA models}

The PtvARFIMA model is first introduced by Amimour and Belaide \cite{AB2020-2}, this model
extends the ARFIMA model from hosking \cite{H1981}. PtvARFIMA time series exhibit long
range dependence (LRD), and the periodic phenomenon, the periodic covariance
function $\gamma _{X}^{i}(h)$ echoes LRD and decays periodically as
the lag $h$ increases, the asymptotic behavior of this function is given by
the following explicit form, as $h\rightarrow \infty$
\begin{equation*}
\gamma _{X}^{i}(h)\simeq \left\{ 
\begin{array}{l}
\begin{array}{l}
\gamma _{X}^{1}(h)\simeq \left\{ 
\begin{array}{c}
\sigma _{1 }^{2}\frac{\Gamma (1-d_{1}-d_{2})}{\Gamma (d_{2})\Gamma
	(1-d_{2})}(h)^{d_{1}+d_{2}-1}, \text{ if }h\equiv 1[p], \\ 
\sigma _{1 }^{2}\frac{\Gamma (1-d_{1}-d_{3})}{\Gamma (d_{3})\Gamma
	(1-d_{3})}(h)^{d_{1}+d_{3}-1}, \text{if }h\equiv 2[p], \\ 
. \\ 
. \\ 
\sigma _{1 }^{2}\frac{\Gamma (1-2d_{1})}{\Gamma
	(d_{1})\Gamma (1-d_{1})}(h)^{2d_{1}-1}, \text{ if }h\equiv 0[p].%
\end{array}%
\right. \\ 
\gamma _{X}^{2}(h)\simeq \left\{ 
\begin{array}{c}
\sigma _{2 }^{2}\frac{\Gamma (1-d_{2}-d_{3})}{\Gamma (d_{3})\Gamma
	(1-d_{3})}(h)^{d_{2}+d_{3}-1}, \text{ if }h\equiv 1[p], \\ 
\sigma _{2 }^{2}\frac{\Gamma (1-d_{2}-d_{4})}{\Gamma (d_{4})\Gamma
	(1-d_{4})}(h)^{d_{2}+d_{4}-1}, \text{if }h\equiv 2[p], \\ 
. \\ 
. \\ 
\sigma _{2 }^{2}\frac{\Gamma (1-2d_{2})}{\Gamma (d_{2})\Gamma
	(1-d_{2})}(h)^{2d_{2}-1}, \text{ if }h\equiv 0[p].%
\end{array}%
\right.%
\end{array}
\\ 
. \\ 
. \\ 
. \\ 
\gamma _{X}^{p}(h)\simeq \left\{ 
\begin{array}{c}
\sigma _{p}^{2}\frac{\Gamma (1-d_{p}-d_{1})}{\Gamma (d_{1})\Gamma
	(1-d_{1})}(h)^{d_{p}+d_{1}-1}, \text{ if }h\equiv 1[p], \\ 
\sigma _{p }^{2}\frac{\Gamma (1-d_{p}-d_{2})}{\Gamma
	(d_{2})\Gamma (1-d_{2})}(h)^{d_{2}+d_{p}-1}, \text{ if }h\equiv 2[p], \\ 
. \\ 
. \\ 
\sigma _{p}^{2}\frac{\Gamma (1-2d_{p})}{\Gamma
	(d_{p})\Gamma (1-d_{p})}(h)^{2d_{p}-1}, \text{ if }h\equiv
0[p].%
\end{array}%
\right.%
\end{array}%
\right. .
\end{equation*}

In this section, we will focus on the PtvARFIMA model as a multivariate
stationary ARFIMA time series, that is, $%
\mathbb{R}
^{p}-$valued time series $X_{m}=(X_{1+pm},....,X_{p+pm})^{^{T}},$ where $T$
denotes the transpose.

\begin{proposition}
	A PtvARFIMA model defined in (\ref{eq0.4}) is a $p$-variate second order stationary long
	memory time series.
\end{proposition}

\begin{proof}
	Suppose that for $a>0,$ a diagonal matrix $N=diag(n_{1},....,n_{p}),$ we
	write $a^{N}=diag(a^{n_{1}},.....,a^{n_{p}}).$ The periodic covariance
	function $\gamma _{X}^{i}(h)$ previously given, can be expressed as a matrix function of a
	second order stationary series $X=\left\{ X_{m}\right\} _{m\in 
		\mathbb{Z}
	},$ given by
	
	\begin{equation}
	\gamma _{X}^{i,k}(h)\simeq \sigma _{i}^{2} h^{D-\frac{1}{2}I}Rh^{D-\frac{1}{2}I}=\left(\sigma _{i}^{2}
	R^{i,k}h^{d_{i}+d_{k}-1}\right) ,\text{ }i,k=1,..,p,\text{ as }h\rightarrow
	\infty ,
	\end{equation}
	
	where $D=diag(d_{1},....,d_{p}),$ and $R^{i,k}>0$ for all $i,k=1,....,p$.
\end{proof}

\begin{proposition}
	\bigskip The causal representation of the $p-$variate ARFIMA time series,
	associated with the PtvARFIMA model with period $p$ is
	
	\begin{equation}
	X_{m}=\underset{j=0}{\overset{\infty }{\sum }}\Upsilon _{j}\varepsilon
	_{m-j},
	\end{equation}
	
	and 
	\begin{equation}
	R^{i,k}=\frac{\Gamma (d_{i})\Gamma (d_{k})}{\Gamma (d_{i}+d_{k})}\left(
	v_{i}v_{k}\frac{\sin (\pi d_{k})}{\sin (\pi (d_{i}+d_{k}))}\right) ,
	\end{equation}
	where $\Upsilon _{j}$ is a diagonal matrix $diag(\psi _{j}^{1},.....,\psi
	_{j}^{p})$ and $\varepsilon _{m}$ is an $%
	\mathbb{R}
	^{p}-$valued white noise\\ $(\varepsilon _{1+pm},....,\varepsilon
	_{p+pm})^{^{T}}$, satisfying $\mathbb{E}[\varepsilon _{m}]=0$ and $\mathbb{E}[\varepsilon
	_{m}\varepsilon _{m}^{^{\prime }}]=\sigma _{i}^{2}I$.
\end{proposition}

\begin{proof}
	Let us calculate the autocovariance function $\gamma _{X}^{i,k}(h)$ 
	\begin{eqnarray*}
		\gamma _{X}^{i,k}(h) &\simeq &\sigma _{i}^{2}\underset{j=0}{\overset{\infty }{\sum }}%
		v_{i}j^{d_{i}-1}v_{k}(j+h)^{d_{k}-1} \\
		&\simeq &\sigma _{i}^{2}\underset{j=0}{\overset{\infty }{\sum }}%
		v_{i}v_{k}j^{d_{i}-1}(j+h)^{d_{k}-1} \\
		&\simeq &\sigma _{i}^{2}\underset{j=0}{\overset{\infty }{\sum }}%
		v_{i}v_{k}j^{d_{i}-1}(j+h)^{d_{k}-1} \\
		&\simeq &\sigma _{i}^{2}\underset{j=0}{\overset{\infty }{\sum }}%
		v_{i}v_{k}j^{d_{i}-1}(j+h)^{d_{k}-1}=\underset{j=0}{\overset{\infty }{\sum }}%
		v_{i}v_{k}h^{d_{i}-1}\left( \frac{j}{h}\right) ^{d_{i}-1}h^{d_{k}-1}\left( 
		\frac{j}{h}+1\right) ^{d_{k}-1},
	\end{eqnarray*}
	from (Appendix A, in Kechagias and Pipiras \cite{KP2015})
	\begin{equation*}
	\gamma _{X}^{i,k}(h)\simeq \sigma _{i}^{2}h^{d_{i}+d_{k}+1}v_{i}v_{k}\int_{1}^{\infty
	}z^{d_{i}-1}(z+1)^{d_{k}-1}dz\sim R^{i,k}h^{d_{i}+d_{k}-1},
	\end{equation*}
	where 
	\begin{equation*}
	R^{i,k}=v_{i}v_{k}\frac{\Gamma (d_{i})\Gamma (1-d_{i}-d_{k})}{\Gamma
		(1-d_{k})},
	\end{equation*}
	and using the identity 
	\begin{equation*}
	\Gamma (z)\Gamma (1-z)=\frac{\pi }{\sin (\pi z)},\text{ }0<z<1,\text{ }
	\end{equation*}
	
	we deduce the expression of $R^{i,k}.$
\end{proof}
\section{A new infinite autoregressive representation for PtvARFIMA }
\begin{theorem}
	\label{eq3.1}
	Let $\{X_{i+pm}\}_{m\in 
		\mathbb{Z}
	}$ be the PARFIMA$(0,d_{i},0)$ process, given by the expression (\ref{eq0.4}), then
	
	(i) When $-\frac{1}{2}<d_{i}<\frac{3}{2}$, or $|d_{i}|<1$.  $%
	\{X_{i+pm}\}_{m\in 
		\mathbb{Z}
	}$ is a invertible process with infinite autoregressive representation
	given by
	
	\begin{equation}
	\label{3.2}
	\varepsilon _{i+pm}=\overset{\infty }{\underset{j=0}{\sum }}\Pi
	_{j}(i)X_{i+pm-j}%
	\makeatletter
	\renewcommand\theequation{\thesection.\arabic{equation}}
	\@addtoreset{equation}{section}
	\makeatother%
	.
	\end{equation}
	The coefficients $\Pi _{j}(i)$ $j\geq 1$ in (\ref{eq3.1}) satisfy the following
	relation 
	\begin{equation}
	\Pi _{j}(i)=-(\psi _{j}(i)+\underset{l=1}{\overset{j-1}{\sum }}\Pi
	_{l}(i)\psi _{j-l}(i+k)),%
	\makeatletter
	\renewcommand\theequation{\thesection.\arabic{equation}}
	\@addtoreset{equation}{section}
	\makeatother%
	\end{equation}
	where the weights $\psi _{j}(i)$ $j\geq 0$ are given in (\ref{eq0.6}) and $l\equiv
	k[p]$, $\Pi _{0}(i)=1,$ with $\overset{\infty }{\underset{j=0}{\sum }}|\Pi
	_{j}(i)|<\infty $. 
	
	(ii) When $-\frac{3}{2}<d_{i}<\frac{1}{2}$, or $|d_{i}|<1$. 
	$\{X_{i+pm}\}_{m\in 
		\mathbb{Z}
	}$ is a causal process with infinite moving average representation given by
	
	\begin{equation}
	\label{eq3.3}
	X_{i+pm}=\overset{\infty }{\underset{j=0}{\sum }}\Psi _{j}(i)\varepsilon
	_{i+pm-j}%
	\makeatletter
	\renewcommand\theequation{\thesection.\arabic{equation}}
	\@addtoreset{equation}{section}
	\makeatother%
	.
	\end{equation}
	The coefficients $\Psi _{j}(i)$ $j\geq 1$ in (\ref{eq3.3}) satisfy the following
	relation 
	\begin{equation}
	\Psi _{j}(i)=-(\pi _{j}(i)+\underset{l=1}{\overset{j-1}{\sum }}\Psi
	_{l}(i)\pi _{j-l}(i+k)),%
	\makeatletter
	\renewcommand\theequation{\thesection.\arabic{equation}}
	\@addtoreset{equation}{section}
	\makeatother%
	\end{equation}
	where the weights $\pi _{j}(i)$ $j\geq 0$ are given in (\ref{eq0.5}) and $l\equiv
	k[p]$, $\Pi _{0}(i)=1,$ with $\overset{\infty }{\underset{j=0}{\sum }}|\Psi
	_{j}(i)|<\infty $.
\end{theorem}

\begin{proof}
	For simplicity of notation, we use $\psi _{j}(i)=\psi _{j}^{i}$ and $\pi _{j}(i)=\pi _{j}^{i}$.\\
	$(i)$ We have
	\begin{equation}
	X_{i+pm}=\overset{\infty }{\underset{j=0}{\sum }}\psi _{j}(i)\varepsilon
	_{i+pm-j},%
	\makeatletter
	\renewcommand\theequation{\thesection.\arabic{equation}}
	\@addtoreset{equation}{section}
	\makeatother%
	\end{equation}
	we can develop the series as follows	
	\begin{eqnarray*}
		X_{i+pm} &=&\varepsilon _{i+pm}+\psi _{1}(i)\varepsilon _{i+pm-1}+\psi
		_{2}(i)\varepsilon _{i+pm-2}+.............. \\
		&=&\varepsilon _{i+pm} \\
		&&+\psi _{1}(i)[\overset{X_{i+pm-1}}{\overbrace{\varepsilon _{i+pm-1}+\psi
				_{1}(i+1)\varepsilon _{i+pm-2}+\psi _{2}(i+1)\varepsilon _{i+pm-3}......}}]
		\\
		&&-\psi _{1}(i)[\psi _{1}(i+1)\varepsilon _{i+pm-2}+\psi
		_{2}(i+1)\varepsilon _{i+pm-3}............] \\
		&&+\psi _{2}(i)\varepsilon _{i+pm-2}+\psi _{3}(i)\varepsilon
		_{i+pm-3}...............
	\end{eqnarray*}
	
	\begin{eqnarray*}
		X_{i+pm}-\psi _{1}(i)X_{i+pm-1} &=&\varepsilon _{i+pm} \\
		&&+[\psi _{2}(i)-\psi _{1}(i)\psi _{1}(i+1)][\overset{X_{i+pm-2}}{\overbrace{%
				\varepsilon _{i+pm-2}+\psi _{1}(i+2)\varepsilon _{i+pm-3}+.......}}] \\
		&&-[\psi _{2}(i)-\psi _{1}(i)\psi _{1}(i+1)][\psi _{1}(i+2)\varepsilon
		_{i+pm-3}............] \\
		&&+[\psi _{3}(i)-\psi _{1}(i)\psi _{2}(i+1)]\varepsilon
		_{i+pm-3}+...............
	\end{eqnarray*}
	
	\begin{eqnarray*}
		X_{i+pm}-\overset{\phi _{1}(i)}{\overbrace{\psi _{1}(i)}}X_{i+pm-1}
		&=&\varepsilon _{i+pm} \\
		&&+[\overset{\phi _{2}(i)}{\overbrace{\psi _{2}(i)-\psi _{1}(i)\psi _{1}(i+1)%
			}}][\overset{X_{i+pm-2}}{\overbrace{\varepsilon _{i+pm-2}+\psi
			_{1}(i+2)\varepsilon _{i+pm-3}+.......}}] \\
	&&-[\psi _{2}(i)-\psi _{1}(i)\psi _{1}(i+1)][\psi _{1}(i+2)\varepsilon
	_{i+pm-3}............] \\
	&&+[\psi _{3}(i)-\psi _{1}(i)\psi _{2}(i+1)]\varepsilon
	_{i+pm-3}+...............
\end{eqnarray*}

\begin{equation}
X_{i+pm}-\phi _{1}(i)X_{i+pm-1}-\phi
_{2}(i)X_{i+pm-2}................=\varepsilon _{i+pm},%
\makeatletter
\renewcommand\theequation{\thesection.\arabic{equation}}
\@addtoreset{equation}{section}
\makeatother%
\end{equation}
such that 
\begin{equation}
\phi _{j}(i)=\psi _{j}(i)-\overset{j-1}{\underset{l=1}{\sum }}\phi
_{l}(i)\psi _{j-l}(i+k),%
\makeatletter
\renewcommand\theequation{\thesection.\arabic{equation}}
\@addtoreset{equation}{section}
\makeatother%
\end{equation}
with $l\equiv k[p].$ Next it is easy to show that $\Pi _{j}(i)=-\phi _{j}(i).$ This result is illustrated by the simulation in the next section.

$(ii)$ The proof is similar to (i) but $d_{i}$ must be replaced by $-d_{i}.$
\end{proof}

\section{Simulation}

In this section, we provide some simple examples, to illustrate the above theoretical results in the invertible case, as well as the non-invertible case for $p=2$. More precisely, we show that the absolute sum of coefficients $\Pi _{j}(i)$ is finite for $-\frac{1}{2}<d_{i}<\frac{3}{2}$ or $|d_{i}|<1$ in the first case and infinite for $d_{1}<-\frac{1}{2}$ and $d_{2}<\frac{3}{2}$,	$d_{1}>-\frac{1}{2}$ and $d_{2}>\frac{3}{2}$ or $|d_{i}|>1$ in the second case, for all $N$ such that $j=0,1,.....,N$. The sum of $\Pi _{j}(i)$ is summable absolutely if the difference between
the previous and the present value of $\Pi _{j}(i)$, e.g. $(|\Pi _{N}(i)-\Pi
_{N+1}(i)|)$, will be very low or negligible. Numerical values of  $|\Pi _{N}(i)-\Pi_{N+1}(i)|$ are presented in table (1) for the invertible case. Table (2) for non invertible case, we find that the results indicated in the tables reflect the theoretical framework.

\begin{table}[]
	\centering
	\begin{tabular}{|ll|l|l|l|l|l|l|l|l|l|}
		\hline
		\multicolumn{1}{|l|}{$N$} & $d_{1}$ & $d_{2}$ & $d_{1}$ & $d_{2}$ & $d_{1}$
		& $d_{2}$ & $d_{1}$ & $d_{2}$ & $d_{1}$ & $d_{2}$ \\ \cline{2-11}
		\multicolumn{1}{|l|}{} & $0.15$ & $0.8$ & $1.49$ & $-0.49$ & $0.75$ & $0.2$
		& $0.9$ & $0.09$ & $-0.2$ & $-0.7$ \\ \hline
		$10$ & \multicolumn{2}{|l}{$1.976584\times 10^{-03}$} & \multicolumn{2}{|l}{$%
			4.02289467$} & \multicolumn{2}{|l}{$0.0128260513$} & \multicolumn{2}{|l}{$%
			6.675804\times 10^{-03}$} & \multicolumn{2}{|l|}{$2.775860\times -04$} \\ 
		\hline
		$25$ & \multicolumn{2}{|l}{$7.551323\times 10^{-04}$} & \multicolumn{2}{|l}{$%
			1.04204572$} & \multicolumn{2}{|l}{$0.0016859353$} & \multicolumn{2}{|l}{$%
			6.473788\times 10^{-04}$} & \multicolumn{2}{|l|}{$1.642120\times -04$} \\ 
		\hline
		$50$ & \multicolumn{2}{|l}{$2.060515\times 10^{-04}$} & \multicolumn{2}{|l}{$%
			0.71213528$} & \multicolumn{2}{|l}{$0.0007002542$} & \multicolumn{2}{|l}{$%
			2.067307\times 10^{-04}$} & \multicolumn{2}{|l|}{$5.866394\times -05$} \\ 
		\hline
		$75$ & \multicolumn{2}{|l}{$1.214109\times 10^{-04}$} & \multicolumn{2}{|l}{$%
			0.50327845$} & \multicolumn{2}{|l}{$0.0003220433$} & \multicolumn{2}{|l}{$%
			8.734869\times 10^{-05}$} & \multicolumn{2}{|l|}{$4.469138\times -05$} \\ 
		\hline
		$100$ & \multicolumn{2}{|l}{$7.671041\times 10^{-05}$} & \multicolumn{2}{|l}{%
			$0.09702500$} & \multicolumn{2}{|l}{$0.0001664161$} & \multicolumn{2}{|l}{$%
			4.752873\times 10^{-05}$} & \multicolumn{2}{|l|}{$3.821700\times -05$} \\ 
		\hline
		\multicolumn{1}{|l|}{$N$} & $d_{1}$ & $d_{2}$ & $d_{1}$ & $d_{2}$ & $d_{1}$
		& $d_{2}$ & $d_{1}$ & $d_{2}$ & $d_{1}$ & $d_{2}$ \\ \cline{2-11}
		\multicolumn{1}{|l|}{} & $-0.6$ & $-0.3$ & $-0.9$ & $-0.09$ & $-0.5$ & $-0.4$
		& $-0.49$ & $-0.9$ & $0.49$ & $0.09$ \\ \hline
		$10$ & \multicolumn{2}{|l}{$0.0252137$} & \multicolumn{2}{|l}{$4.0228946$} & 
		\multicolumn{2}{|l}{$0.01134118$} & \multicolumn{2}{|l}{$2.696864\times
			10^{-03}$} & \multicolumn{2}{|l|}{$0.01195167$} \\ \hline
		$25$ & \multicolumn{2}{|l}{$0.0092933$} & \multicolumn{2}{|l}{$1.0420457$} & 
		\multicolumn{2}{|l}{$0.00540992$} & \multicolumn{2}{|l}{$4.487907\times
			10^{-04}$} & \multicolumn{2}{|l|}{$0.00190099$} \\ \hline
		$50$ & \multicolumn{2}{|l}{$0.0061061$} & \multicolumn{2}{|l}{$0.7121352$} & 
		\multicolumn{2}{|l}{$0.00242034$} & \multicolumn{2}{|l}{$9.400265\times
			10^{-05}$} & \multicolumn{2}{|l|}{$0.00097871$} \\ \hline
		$75$ & \multicolumn{2}{|l}{$0.0040766$} & \multicolumn{2}{|l}{$0.5032784$} & 
		\multicolumn{2}{|l}{$0.00173045$} & \multicolumn{2}{|l}{$4.431224\times
			10^{-05}$} & \multicolumn{2}{|l|}{$0.00047477$} \\ \hline
		$100$ & \multicolumn{2}{|l}{$0.0032788$} & \multicolumn{2}{|l}{$0.0970250$}
		& \multicolumn{2}{|l}{$0.00131972$} & \multicolumn{2}{|l}{$2.515928\times
			10^{-05}$} & \multicolumn{2}{|l|}{$0.00032591$} \\ \hline
	\end{tabular}%
	\caption{ Values of $(|\Pi _{N}(i)-\Pi
		_{N+1}(i)|)$ for $-\frac{1}{2}<d_{i}<\frac{3}{2}$ or $|d_{i}|<1$ and various $N$}
	\label{tab1}
\end{table}

\begin{table}[]
	\centering
	\begin{tabular}{|ll|l|l|l|l|l|}
		\hline
		\multicolumn{1}{|l|}{$N$} & $d_{1}$ & $d_{2}$ & $d_{1}$ & $d_{2}$ & $d_{1}$
		& $d_{2}$ \\ \cline{2-7}
		\multicolumn{1}{|l|}{} & $-0.6$ & $1.49$ & $-0.4$ & $1.65$ & $-1.2$ & $-1.4$
		\\ \hline
		$10$ & \multicolumn{2}{|l}{$1.4451672$} & \multicolumn{2}{|l}{$0.7828075$} & 
		\multicolumn{2}{|l|}{$0.2448915$} \\ \hline
		$25$ & \multicolumn{2}{|l}{$11.2747847$} & \multicolumn{2}{|l}{$4.0371864$}
		& \multicolumn{2}{|l|}{$0.4049227$} \\ \hline
		$50$ & \multicolumn{2}{|l}{$21.7021959$} & \multicolumn{2}{|l}{$3.6322384$}
		& \multicolumn{2}{|l|}{$0.4341756$} \\ \hline
		$75$ & \multicolumn{2}{|l}{$332.0970990$} & \multicolumn{2}{|l}{$27.3838075$}
		& \multicolumn{2}{|l|}{$0.5108372$} \\ \hline
		$100$ & \multicolumn{2}{|l}{$1686.6896433$} & \multicolumn{2}{|l}{$68.7898114
			$} & \multicolumn{2}{|l|}{$0.5445771$} \\ \hline
	\end{tabular}%
	\caption{Values of $(|\Pi _{N}(i)-\Pi
		_{N+1}(i)|)$ for $d_{1}<-\frac{1}{2}$ and $d_{2}<\frac{3}{2}$,	$d_{1}>-\frac{1}{2}$ and $d_{2}>\frac{3}{2}$ or $|d_{i}|>1$ and various $N$}
	\label{tab2}
\end{table}

For $-\frac{1}{2}<d_{i}<\frac{3}{2}$ or $|d_{i}|<1$ (see table \ref{tab1})
it is easy to see that the sum of $|\Pi _{j}(i)|$ converges, i.e the sum
of $|\Pi _{j}(i)|$ stabilizes quickly in this case, since the quantity  $|\Pi _{N}(i)-\Pi _{N+1}(i)|$  decrease when $N$ increase and have
inconsiderable values. Contrarily, if the values of $d_{i}$ are outside the intervals cited (see table \ref{tab2}), this demonstrates that the sum diverges even for large values of $N$, in this case the
quantity $|\Pi _{N}(i)-\Pi _{N+1}(i)|$ is considerable and sometimes increases when $N$ increase.

\section{Conclusion}
\label{sec:conc}

In this article, we have discussed in one hand the p-variate stationary ARFIMA models associated with the PARFIMA model considered here. We have represented the model as a causal representation, showing that the asymptotic behavior of the periodic autocovariance function can be obtained using the multivariate moving average representation. On the other hand, we have established a new infinite autoregressive representation for the PARFIMA model and illustrated the results via a simulation study.

\end{document}